\documentclass[a4paper,12pt]{article}

\usepackage{amsfonts,amsthm,amssymb,amsmath}
\usepackage[colorlinks=true,linkcolor=blue,citecolor=blue]{hyperref}
\usepackage{fourier}

\newtheorem{theorem}{Theorem}[section]
\newtheorem{proposition}[theorem]{Proposition}
\newtheorem{corollary}[theorem]{Corollary}
\newtheorem{lemma}[theorem]{Lemma}
\theoremstyle{definition}

\newtheorem{remark}[theorem]{Remark}

\DeclareMathOperator{\mon}{mon}
\DeclareMathOperator{\ct}{Cot}
\DeclareMathOperator{\re}{Re}

\title{Bohr's absolute convergence problem for $\mathcal{H}_p$-Dirichlet series in Banach spaces}

\author{Daniel Carando\footnote{Departamento de Matematica. Universidad de Buenos Aires. 1428 Buenos Aires and IMAS - CONICET (Argentina)}
\and
Andreas Defant\footnote{Institut f\"ur Mathematik. Universit\"at Oldenburg. D-26111 Oldenburg (Germany) }
\and
Pablo Sevilla-Peris  \footnote{Instituto Universitario de Matem\'atica Pura y Aplicada. Universitat Polit\`ecnica de Val\`encia. 46022 Valencia (Spain)}}

\date{}

\begin{document}
\maketitle

\footnotetext{The first author was partially supported by CONICET PIP 0624, PICT 2011-1456 and UBACyT 1-746. The second and third authors
were supported by MICINN Project MTM2011-22417}


\begin{abstract}
The Bohr-Bohnenblust-Hille Theorem states that the width of the strip in the complex plane on which an ordinary  Dirichlet series
$\sum_n a_n n^{-s}$ converges uniformly but not absolutely is less than or equal to $1/2$, and this estimate is optimal. Equivalently,
the supremum of the absolute convergence abscissas of all Dirichlet series in the Hardy space $\mathcal{H}_\infty$  equals $1/2$. By a surprising fact of  Bayart the same result holds true if $\mathcal{H}_\infty$ is replaced by any Hardy space $\mathcal{H}_p$, $1 \le p < \infty$, of Dirichlet series. For Dirichlet series with coefficients in a Banach space $X$ the maximal width of Bohr's strips depend on the geometry of $X$; Defant, Garc\'ia, Maestre and P\'erez-Garc\'ia proved that such maximal width  equal $1- 1/\ct(X)$, where $\ct(X)$ denotes the maximal cotype of $X$. Equivalently, the supremum over the absolute convergence abscissas of all Dirichlet series in the vector-valued Hardy space $\mathcal{H}_\infty(X)$  equals $1- 1/\ct(X)$. In this article we show that this
result remains true if  $\mathcal{H}_\infty(X)$ is replaced by the larger class $\mathcal{H}_p(X)$, $1 \le p < \infty$.

\end{abstract}

\section{Main result and its motivation }

Given a Banach space $X$, an ordinary Dirichlet series in $X$  is a series of the form $D=\sum_n a_n n^{-s}$, where the coefficients $a_n$ are vectors in $X$ and $s$ is a complex variable. Maximal domains
where such Dirichlet series converge conditionally, uniformly or absolutely are half planes $[\re  > \sigma]$, where $\sigma=\sigma_c,\sigma_u$ or $\sigma_a$ are called the abscissa of conditional, uniform or absolute convergence, respectively. More precisely,
$\sigma_\alpha (D)$ is the infimum of all $r\in\mathbb{R}$ such that  on $[\re  >r]$ we have convergence of $D$ of the requested type $\alpha = c,u$ or $a$. Clearly, we have
$\sigma_c (D) \le \sigma_u (D) \le \sigma_a (D)$,
and it can be easily shown that $\sup \sigma_a (D) - \sigma_c (D) =1$\,,
where the supremum is taken over all Dirichlet series $D$ with coefficients in $X$. To determine the
maximal width of the strip on which a Dirichlet series in $X$ converges uniformly but not absolutely,
is more complicated. The main result of \cite{DeGaMaPG08} states, with the notation given below, that
\begin{equation} \label{MathAnn}
S(X):= \sup \sigma_a (D) - \sigma_u (D) = 1 - \frac{1}{\ct(X)}\,.
\end{equation}
Recall that a Banach space $X$ is of cotype $q$, $2 \le q < \infty$ whenever there is a constant $C \ge 0$
such that for each choice of finitely many vectors $x_1, \ldots, x_N \in X$ we have
\begin{align}\label{cotype}
\Big(\sum_{k=1}^N \big\| x_k \big\|_X^q\Big)^{1/q}  \leq C \Big( \int_{\mathbb{T}^N} \Big\| \sum_{k=1}^N  x_k z_k  \Big\|_X^2 dz \Big)^{1/2}\,,
\end{align}
where $\mathbb{T}:= \big\{z \in \mathbb{C}\,\big|\, |z|=1\big\}$  and $\mathbb{T}^N$ is endowed with $N$th product of the normalized Lebesgue measure on $\mathbb{T}$. We denote by $C_r(X)$ the best of such constants $C$. As usual we write
\[
\ct(X) := \inf \Big\{2 \le q< \infty \,\big|\, X \,\,\, \text{cotype} \,\,\, q \Big\}\,,
\]
and (although this infimum in general is not attained) we call it the optimal cotype of $X$.
{If there is no $2\le q <\infty $ for which $X$ has cotype $q$,
then $X$ is said to have no finite cotype, and we put $\ct(X)= \infty$. }
To see an example,
\[
\ct(X)(\ell_q)
  =
\begin{cases}
  q & \text{ for } 2 \leq q \leq \infty \\
  2 & \text{ for } 1\le  q \le 2\,.
\end{cases}
\]
The scalar case $X= \mathbb{C}$ in \eqref{MathAnn} was first studied by Bohr and Bohnenblust-Hille: In 1913 Bohr in \cite{Bo13_Goett} proved that $S(\mathbb{C}) \leq \frac{1}{2}$, and in 1931 Bohnenblust and Hille
in \cite{BoHi31}  that $S(\mathbb{C}) \ge \frac{1}{2}$. Clearly, the equality
\begin{equation} \label{BBH}
S(\mathbb{C}) = \frac{1}{2}\,,
\end{equation}
 nowadays called   {\it Bohr-Bohnenblust-Hille Theorem},
fits with \eqref{MathAnn}.
Let us give a second formulation of \eqref{MathAnn}. Define the vector space   $\mathcal{H}_\infty(X)$
 of all Dirichlet series $D = \sum_n a_n n^{-s}$ in $X$ such that
 \begin{itemize}
 \item $\sigma_c(D) \leq 0$\,,
  \item the function $D(s)= \sum_n a_n \frac{1}{n^s}$ on $ \re s >0$ is bounded.
  \end{itemize}
  Then $\mathcal{H}_\infty(X)$ together with the norm
\[
\|D\|_{\mathcal{H}_\infty(X)} =\sup_{\re s > 0}  \Big\|\sum_{n=1}^\infty a_n \frac{1}{n^s} \Big\|_X
\]
forms a Banach space. For any Dirichlet series $D$ in $X$ we have
\begin{equation} \label{BohrFUND}
\sigma_u(D) = \inf \Big\{ \sigma \in \mathbb{R} \,\,\, \big|\, \,\, \sum_n \frac{a_n}{n^\sigma}\frac{1}{n^s} \in \mathcal{H}_\infty(X) \Big\} \,.
\end{equation}
In the scalar case $X=\mathbb{C}$, this is (what we call) {\it Bohr's fundamental theorem}  from \cite{Bo13}, and for Dirichlet series in arbitrary
Banach spaces the proof follows  similarly. Together with \eqref{BohrFUND} a simply translation argument  gives the following reformulation of \eqref{MathAnn}:
\begin{equation} \label{MathANN}
S(X)=   \sup_{D \in \mathcal{H}_\infty(X)}\sigma_a(D) =   1 - \frac{1}{\ct(X)}\,.
\end{equation}

Following an ingenious idea of Bohr each Dirichlet series may be identified with a power series in infinitely many variables. More presicely, fix a Banach space $X$ and denote by $\mathfrak{P}(X)$  the vector space  of all formal power series
$\sum_{\alpha} c_{\alpha} z^{\alpha}$ in $X$ and by $\mathfrak{D}(X)$ the vector space  of all Dirichlet series $\sum_n a_{n} n^{-s}$ in $X$ . Let as usual  $(p_{n})_{n}$ be the sequence of prime numbers. Since each integer $n$ has a unique prime number decomposition
$n=p_{1}^{\alpha_{1}} \cdots p_{k}^{\alpha_{k}} = p^{\alpha}$ with $\alpha_j \in \mathbb{N}_0, \, 1 \le j \leq k$, the linear mapping
\begin{align} \label{vision}
\mathfrak{B}_X: \mathfrak{P}(X) & \longrightarrow \mathfrak{D}(X) \,\,, \,\,\,\,\,\,
\textstyle\sum_{\alpha \in \mathbb{N}_{0}^{(\mathbb{N})}} c_{\alpha} z^{\alpha}
\rightsquigarrow \textstyle\sum_{n=1}^{\infty} a_{n} n^{-s}\,, \,\,\, \text{where}\,\,\, a_{p^{\alpha}} = c_{\alpha}\,
\end{align}
is bijective; we call $\mathfrak{B}_X$ the {\it Bohr transform in $X$}. As discovered by Bayart in \cite{Ba02} this (\textit{a priori} very) formal identification  allows to develop a theory of Hardy spaces of scalar--valued  Dirichlet series.

Similarly we now define  Hardy spaces of $X$--valued  Dirichlet series.
Denote by $dw$ the normalized Lebesgue measure on the  infinite dimensional polytorus $\mathbb{T}^{\infty} = \prod_{k=1}^\infty \mathbb{T}$, e.g. the countable product measure  of the normalized Lebesgue measure
on $\mathbb{T}$. For any  multi index $\alpha = (\alpha_1, \dots, \alpha_n,0, \ldots ) \in \mathbb{Z}^{(\mathbb{N})}$ (all finite sequences in $\mathbb{Z}$)
the $\alpha$th Fourier coefficient $\hat{f}(\alpha)$ of
  $f  \in L_{1} (\mathbb{T}^{\infty}, X)$ is given by
\[
\hat{f}(\alpha) = \int_{\mathbb{T}^{\infty}} f(w) w^{- \alpha} dw\,,
\]
 where we as usual write $w^{\alpha}$ for the monomial $w_{1}^{\alpha_{1}}\ldots w_n^{\alpha_n}$.
Then, given $1 \le p < \infty$, the $X$-valued Hardy space on $\mathbb{T}^\infty$ is the subspace of $L_{p} (\mathbb{T}^{\infty}, X)$ defined as
\begin{align}
H_{p}(\mathbb{T}^{\infty},X) = \Big\{ f \in  L_{p} (\mathbb{T}^{\infty},X) \,\,\big|\,\, \hat{f}(\alpha) = 0 \,  , \, \,\, \,
\forall \alpha \in \mathbb{Z}^{(\mathbb{N})} \setminus \mathbb{N}_{0}^{(\mathbb{N})} \Big\}.
\end{align}
Assigning to each $f \in H_{p}(\mathbb{T}^{\infty},X)$ its unique formal power series $\sum_\alpha \hat{f}(\alpha) z^\alpha$ we may consider $H_{p}(\mathbb{T}^{\infty},X)$ as a subspace of $\mathfrak{P}(X)$. We denote the  image
of this subspace under the Bohr transform $\mathfrak{B}_X$ by
\[
\mathcal{H}_p(X)\,.
\]
This vector space of all (so-called) $\mathcal{H}_p(X)$-Dirichlet series $D$  together with the norm
$$\|D\|_{\mathcal{H}_p(X)}= \|\mathfrak{B}_X^{-1}(D)\|_{H_{p}(\mathbb{T}^{\infty},X)}$$
forms a Banach space; in other words, through Bohr's transform $\mathfrak{B}_X$ from \eqref{vision} we  by definition  identify
\[
\mathcal{H}_p(X) = H_{p}(\mathbb{T}^{\infty},X)\,, 1 \le p < \infty.
\]
For $p = \infty$  we this way of course could also define a Banach space $\mathcal{H}_\infty(X)$, and it turns out that at least in the scalar case  $X= \mathbb{C}$ this definition then coincides with the one given above; but we remark
that these two $\mathcal{H}_\infty(X)$'s are different for arbitrary $X$.
It is important to note that by the Birkhoff-Khinchine ergodic theorem  the  following internal description of the $\mathcal{H}_p(X)$-norm for finite Dirichlet polynomials
$D= \sum_{k=1}^n a_k n^{-s}$ holds:
\[
\|D\|_{\mathcal{H}_p(X)}=\lim_{T \rightarrow \infty}  \Big(\frac{1}{2T}\int_{-T}^T \Big\|\sum_{k=1}^n a_k \frac{1}{n^{t}}\Big\|_X^p dt  \Big)^{1/p}
\]
(see e.g. Bayart \cite{Ba02} for the scalar case, and the vector-valued case follows exactly the same way).

Motivated by  \eqref{BohrFUND} we define for  $D\in \mathfrak{D}(X)$ and $1 \le p <\infty$
\[
\sigma_{\mathcal{H}_p(X)}(D): =
\inf \Big\{ \sigma \in \mathbb{R} \,\, \big| \,\, \sum_n \frac{a_n}{n^\sigma}\frac{1}{n^s} \in \mathcal{H}_p(X) \Big\}\,,
\]
and motivated by \eqref{MathANN} we define
\[
S_{p}(X):= \sup_{D \in \mathfrak{D}(X)} \sigma_a (D) - \sigma_{\mathcal{H}_p(X)} (D) = \sup_{D \in \mathcal{H}_p(X)}\sigma_a(D) \,
\]
 (for the second equality use again a simple  translation argument). A result of Bayart \cite{Ba02} shows that for every $1 \leq p < \infty$
\begin{align}\label{Bayart?}
S_{p}(\mathbb{C}) = \frac{1}{2}\,,
\end{align}
which according to Helson \cite{He05} is a bit surprising since $\mathcal{H}_\infty(\mathbb{C})$
is much smaller than $\mathcal{H}_p(\mathbb{C})$.

The following theorem  unifies and generalizes  \eqref{MathAnn}, \eqref{BBH} as well as \eqref{Bayart?}, and it is our main result.
\begin{theorem}\label{MAIN}
For every $1 \leq p \leq \infty$ and every Banach space $X$ we have
\[
S_p(X)= 1- \dfrac{1}{\ct(X)}\,.
\]
\end{theorem}
The proof  will be given in section 4.
But before  we start let us give an interesting reformulation in terms of the monomial convergence of $X$-valued $H_p$-functions on
$\mathbb{T}^{\infty}$.
Fix a Banach space $X$ and $1 \leq p \leq \infty$, and define the set of monomial convergence of
$ H_{p}(\mathbb{T}^{\infty},X)$:
\begin{equation*} \label{mon Hardy}
 \mon H_{p}(\mathbb{T}^{\infty},X) = \Big\{ z \in B_{c_0}\,\,\,\Big|\,\, \, \sum_{\alpha} \| \hat{f}(\alpha) z^{\alpha} \|_X < \infty
\text{ for all }  f \in  H_{p}(\mathbb{T}^{\infty},X) \Big\} \,.
\end{equation*}
Philosophically, this is the  largest set $M$ on which for   each $f \in H_{p}(\mathbb{T}^{\infty},X)$
the definition  $g(z) = \sum_\alpha \hat{f}(\alpha) z^\alpha$, \, $z \in M$ leads to an extension of $f$  from the distinguished boundary $\mathbb{T}^\infty$
to its ``interior'' $B_{c_0}$ (the open unit ball of the Banach space $c_0$ of all null sequences).  For a detailed study of sets of monomial convergence in the scalar case $X= \mathbb{C}$ see \cite{DeMaPr09}, and in the vector-valued case \cite{DeSe11}.

We  later need the following two basic properties of monomial domains (in the scalar case see
 \cite[p.550]{DeGaMaPG08} and \cite[Lemma 4.3]{DeFrMaSe}, and in the vector-valued case the proofs follow similar lines).

\begin{remark} \label{remark}
\text{}
\begin{itemize}
\item[(1)]
Let $z \in \mon H_{p}(\mathbb{T}^{\infty},X)$. Then  $u= (z_{\sigma(n)})_n \in \mon H_{p}(\mathbb{T}^{\infty},X)$
for every permutation $\sigma$ of $\mathbb{N}$.
\item[(2)]
Let $z \in \mon H_{p}(\mathbb{T}^{\infty},X)$  and $x = (x_{n})_{n} \in \mathbb{D}^{\infty}$ be such that $\vert x_{n} \vert \leq \vert z_{n} \vert$ for
all but finitely many $n$'s. Then $x \in  \mon H_{p}(\mathbb{T}^{\infty},X)$.
\end{itemize}
\end{remark}
   Given $1 \le p \leq \infty$ and a Banach space $X$, the following number  measures the size of $\mon H_{p}(\mathbb{T}^{\infty},X)$
within the scale of $\ell_r$-spaces:
\[
M_p(X) = \sup \Big\{ 1 \le r \le \infty \,\, \big| \,\, \ell_r \cap B_{c_0} \subset  \mon H_{p}(\mathbb{T}^{\infty},X)  \Big\}\,.
\]
The following result is a reformulation of Theorem \ref{MAIN} in terms of vector-valued $H_p$-functions on $\mathbb{T}^\infty$ through Bohr's transform $\mathfrak{B}_X$. The proof is modelled along ideas from Bohr's seminal article \cite[Satz IX]{Bo13_Goett}.

\begin{corollary} \label{Bohri}
For each Banach space $X$ and $1 \leq p \leq \infty$ we have
\[
M_p(X) = \frac{\ct(X)}{\ct(X)-1}\,.
\]
\end{corollary}

\begin{proof}
We are going to prove that  $S_p(X)= 1/M_p(X)$, and as a consequence  the conclusion  follows from Theorem \ref{MAIN}.
We begin by showing that $S_p(X)\leq 1/M_p(X)$. We fix $q < M_p(X)$ and $r > 1/q$; then we have that $\big( \frac{1}{p_{n}^{r}} \big)_{n} \in \ell_{q} \cap  B_{c_0}$
and, by the very definition of $M_p(X)$, $\sum_{\alpha} \big\|\hat{f}(\alpha)  \big( \frac{1}{p^{r}} \big)^{\alpha}\big\|_X<\infty$ converges absolutely for every $f \in H_p(\mathbb{T}^\infty,X)$. We choose
now an arbitrary  Dirichlet series
$$
D=\mathfrak{B}_Xf=  \sum_{n}  a_{n}  \frac{1}{n^{r}}  \in \mathcal{H}_p(X)\,\, \text{  with }\,\, f \in H_p(\mathbb{T}^\infty,X)\,.
$$
Then
\[
 \sum_{n} \big\| a_{n} \big\|_X \frac{1}{n^{r}} = \sum_\alpha  \big\| a_{p^{\alpha}} \big\|_X \Big( \frac{1}{p^{\alpha}} \Big)^{r}
= \sum_\alpha  \big\|\hat{f}(\alpha) \big\|_X \Big( \frac{1}{p^{r}} \Big)^{\alpha} < \infty \, .
\]
Clearly, this implies that $S_p(X) \le r$. Since this holds for each  $r > 1/q$, we get that  $S_p(X) \le 1/q$, and since this now holds for each $q < M_p(X)$, we have  $S_p(X) \leq 1/M_p(X)$.
Conversely, let us take some $q > M_p(X)$; then there is  $z \in \ell_{q} \cap  B_{c_0}$ and  $f \in H_\infty(\mathbb{T}^\infty, X)$ such  that $\sum_{\alpha} \hat{f}(\alpha)z^{\alpha}$ does not converge absolutely. By Remark \ref{remark} we may assume that $z$ is decreasing, and hence  $(z_{n} n^{1/q})_{n}$ is bounded.
We choose now $r>q$ and define $w_{n} = \frac{1}{p_{n}^{1/r}}$. By the Prime Number Theorem we know that there is a universal constant
$C>0$ such that
\[
 0< \frac{z_{n}}{w_{n}} = z_{n} p_{n}^{\frac{1}{r}} = z_{n} n^{\frac{1}{q}}\frac{p_{n}^{1/r}}{n^{1/q}}
 = z_{n} n^{\frac{1}{q}}   \Big( \frac{p_{n}}{n} \Big)^{\frac{1}{r}} \frac{1}{n^{1/q-1/r}}
\leq C z_{n} n^{\frac{1}{q}} \frac{(\log n)^{1/r}}{n^{1/q-1/r}} \, .
\]
The last term tends to $0$ as $n \to \infty$; hence $z_{n} \leq w_{n}$ but for a finite number of $n$'s. By Remark \ref{remark} this implies that $\sum_{\alpha} \hat{f}(\alpha)  w^{\alpha}$
does not converge absolutely. But then   $D=\mathfrak{B}_Xf=  \sum_{n}  a_{n}  n^{-r}  \in \mathcal{H}_p(X)$
satisfies
\[
 \sum_{n} \big\| a_{n} \big\|_X \frac{1}{n^{1/r}} = \sum_{\alpha} \big\|a_{p^{\alpha}} \big\|_X \Big( \frac{1}{p^{1/r}} \Big)^{\alpha}
= \sum_{\alpha}\big\|\hat{f}(\alpha) \big\|_X w^{\alpha} = \infty \, .
\]
This gives that $\sigma_a(D)\geq 1/r$ for every $r>q$, hence $ \sigma_a(D)\geq 1/q$. Since  this holds for every $q>M_p(X)$, we finally have $S_p(X) \geq 1/M_p(X)$.
\end{proof}

We shall use standard notation and notions from Banach space
theory, as presented, e.g. in \cite{LiTz77,LiTz79}.
For everything needed on polynomials in Banach spaces see e.g. \cite{Di99} and \cite{Fl99}.

\section{Relevant  inequalities}
The main aim here is to prove a sort of  polynomial extension of the notion of cotype.
Recall the definition of $C_q(X)$ from \eqref{cotype}.
Moreover, from Kahane's inequality  we know that, given $1 \leq q < \infty$, there is a (best) constant $K \ge 1$
such that for each Banach space $X$ and each choice  finitely many vectors  $x_1, \ldots x_N \in X$
\[
\Big( \int_{\mathbb{T}^N} \Big\| \sum_{k=1}^N  x_k z_k  \Big\|_X^2 dz \Big)^{1/2}
\leq  K \int_{\mathbb{T}^N} \Big\| \sum_{k=1}^N  x_k z_k  \Big\|_X dz \,.
\]
As usual we write $|\alpha|=\alpha_1 + \ldots + \alpha_N$
and $\alpha! = \alpha_1! \ldots \alpha_N!$ for every multi index $\alpha \in \mathbb{N}_0^N$.

\begin{proposition} \label{maintool}
Let $X$ be a Banach space of cotype $q$, $2 \leq q < \infty$, and
\[
P: \mathbb{C}^N \rightarrow X,\,\,\, P(z)= \sum_{\substack{\alpha \in \mathbb{N}_0^N \\ |\alpha|=m}} c_\alpha z^\alpha
\]
be an $m$-homogeneous polynomial. Let
\[
T: \mathbb{C}^N \times \ldots \times  \mathbb{C}^N \rightarrow X,\,\,\,
 T(z^{(1)}, \ldots,z^{(m)}) = \sum_{i_1, \ldots,i_m=1}^N a_{i_1, \ldots,i_m} z^{(1)}_{i_1} \ldots z^{(m)}_{i_m}
\]
be the unique $m$-linear symmetrization of $P$. Then
$$
\Big(
\sum_{i_1,\ldots,i_m}
\big\| a_{i_1, \ldots,i_m} \big\|_X^q
\Big)^{1/q}
\leq
\big(C_q(X)\, K \big)^{m} \frac{m^m}{m!}
\,\, \int_{\mathbb{T}^N} \big\| P(z) \big\|_X dz\,.
$$
\end{proposition}
\noindent Before we give the proof let us note that  \cite[Theorem~3.2]{BoPGVi04} is an $m$-linear result that, combined with polarization gives (with the previous notation)
\[
 \Big(
\sum_{i_1,\ldots,i_m}
\big\| a_{i_1, \ldots,i_m} \big\|_X^q
\Big)^{1/q}
\leq C_q(X)^{m} \frac{m^m}{m!} \sup_{z \in \mathbb{D}^N} \Vert P(z) \Vert \,.
\]
Our result allows to replace (up to the constant $K$) the $\Vert \ \Vert_{\infty}$ norm with the smaller norm $\Vert \ \Vert_{1}$. We prepare the proof of Proposition~\ref{maintool} with three lemmas.

\begin{lemma} \label{lemma0}
Let $X$ be a Banach space of cotype $q$, $2 \leq q < \infty.$ Then for every  $m$-linear form
\[
T: \mathbb{C}^N \times \ldots \times  \mathbb{C}^N \rightarrow X,\,\,\,
 T\big(z^{(1)}, \ldots,z^{(m)}\big) = \sum_{i_1, \ldots,i_m=1}^N a_{i_1, \ldots,i_m} z^{(1)}_{i_1} \ldots z^{(m)}_{i_m}
\]
we have
\begin{align*}
\Big(\sum_{i_1, \ldots,i_m=1}^N \big\| a_{i_1, \ldots,i_m} \big\|_X^q\Big)^{1/q}
\leq \big(C_q(X)\,K \big)^m
 \int_{\mathbb{T}^\infty}  \ldots \int_{\mathbb{T}^\infty}
\big\|   T(z^{(1)}, \ldots,z^{(m)})     \big\|_X dz^{(1)} \ldots dz^{(m)}\,.
\end{align*}
\end{lemma}

\begin{proof}
We prove this result by induction on the degree $m$. For $m=1$ the result is an immediate consequence of the definition of cotype $q$ and Kahane's inequality. Assume that the result holds for $m-1$. By the continuous Minkowski inequality we then conclude that for every choice of finitely many vectors
$a_{i_1, \ldots,i_m} \in X$ with $1 \leq i_j \leq N, 1 \leq j \leq m$ we have
\begin{align*}
&
\sum_{i_1, \ldots,i_m} \big\| a_{i_1, \ldots,i_m} \big\|_X^q
=  \sum_{i_1, \ldots,i_{m-1}}\,\sum_{i_m}  \big\| a_{i_1, \ldots,i_{m}} \big\|_X^q
\\&
\leq
C_q(X)^q K^q
\Big(
\sum_{i_1, \ldots,i_{m-1}}
\Big(
\int_{\mathbb{T}^\infty}  \big\| \sum_{i_m} a_{i_1, \ldots,i_{m}} z^{(m)}_{i_m}\big\|_X dz^{(m)}
\Big)^q
\Big)^{q/q}
\\&
\leq
C_q(X)^q K^q
\Big(
\int_{\mathbb{T}^\infty}
 \Big(
 \sum_{i_1, \ldots,i_{m-1}}\big\| \sum_{i_m} a_{i_1, \ldots,i_{m}} z^{(m)}_{i_m}\big\|_X^q
\Big)^{1/q}
dz^{(m)}
\Big)^{q}
\\&
\leq
 C_q(X)^{qm} K^{qm}
\Big(
\int_{\mathbb{T}^\infty}
\underbrace{\int_{\mathbb{T}^\infty}  \ldots \int_{\mathbb{T}^\infty}}_{m-1}
 \big\|
\sum_{i_1, \ldots,i_{m-1}} a_{i_1, \ldots,i_{m-1}} z^{(1)}_{i_1}, \ldots,z^{(m-1)}_{i_{m-1}}
 \big\|_X
   dz^{(1)}\ldots dz^{(m-1)}  dz^{(m)}\Big)^{q}\,,
  \end{align*}
which is the conclusion.
\end{proof}

\noindent The following two lemmas are needed to produce a polynomial analog  of the preceding result.

\begin{lemma}\label{lemma1}
Let $X$ be a Banach space, and $f: \mathbb{C} \rightarrow X$ a holomorphic function. Then for $R_1, R_2,R \ge 0$
with  $R_1+ R_2 \le R$ we have
\[
\int_{\mathbb{T}}  \int_{\mathbb{T}}
\big\| f\big(R_1 z_1 + R_2z_2\big)   \big\|_X dz_1dz_2 \leq
\int_{\mathbb{T}}
\big\| f\big(Rz\big)\big\|_X dz\,.
\]
\end{lemma}

\begin{proof}
By the rotation invariance of the normalized Lebesgue measure on $\mathbb{T}$ we get
\begin{multline*}
\int_{\mathbb{T}} \int_{\mathbb{T}}
\big\| f\big(R_1 z_1 + R_2z_2\big)   \big\|_X dz_1dz_2
=\int_{\mathbb{T}} \int_{\mathbb{T}}
\big\| f\big(R_1 z_1z_2 + R_2z_2\big)   \big\|_X dz_1dz_2
\\
=
\int_{\mathbb{T}} \int_{\mathbb{T}}
\big\| f\big(z_2(R_1 z_1 + R_2)\big)   \big\|_X dz_1dz_2
=
\int_{\mathbb{T}} \int_{\mathbb{T}}
\big\| f\big(z_2|R_1 z_1 + R_2|\big)   \big\|_X dz_2dz_1
\\
=
\int_{\mathbb{T}} \int_{\mathbb{T}}
\big\| f\big(z_2r(z_1)R\big)   \big\|_X dz_2dz_1
=\int_{0}^{2 \pi} \int_{0}^{2 \pi}
\big\| f\big(  r(e^{is}) R e^{it}\big)   \big\|_X \frac{dt}{2 \pi}\,\frac{ds}{2 \pi}\,.
\end{multline*}
where $r(z)= \frac{1}{R} | R_1 z + R_2 |$, $z \in \mathbb{T}$.
We know that for each holomorphic function
$h: \mathbb{C} \rightarrow X$ we have
\[
\int_{\mathbb{T}} \big\|h(z)   \big\|_X dz
=
\sup_{0 \leq r \leq 1} \int_0^{2\pi} \big\|h(r e^{it}) \big\|_X  \frac{dt}{2 \pi}
\]
(see e.g. Blasco and Xu \cite[p. 338]{BlXu91}).
Define now $h(z)= f\big( Rz \big)$, and note that $0 \leq r(z) \leq 1$ for all
$z \in \mathbb{T}$. Then
\begin{align*}
\int_{\mathbb{T}} \int_{\mathbb{T}}
\big\| f\big(R_1 z_1 + R_2z_2\big)   \big\|_X dz_1dz_2
&
=\int_{0}^{2 \pi} \int_{0}^{2 \pi}
\big\| h\big(r(e^{is}) e^{it} \big)   \big\|_X \frac{dt}{2 \pi}\,\frac{ds}{2 \pi}
\\&
\leq
\int_{0}^{2 \pi} \int_{\mathbb{T}}
\big\| h\big(z\big)   \big\|_X dz\,\frac{ds}{2 \pi}
=
 \int_{\mathbb{T}}
\big\| f\big(Rz\big)   \big\|_X dz\,.
\end{align*}
This completes the proof.
\end{proof}

\noindent A sort of iteration of the preceding result leads to the next

\begin{lemma} \label{lemma2}
Let $X$ be a Banach space, and $f: \mathbb{C}^N \rightarrow X$ a holomorphic function. Then for every $m$
\[
\int_{\mathbb{T}^N} \ldots  \int_{\mathbb{T}^N}
\big\| f\big(z^{(1)}+ \ldots + z^{(m)} \big)   \big\|_X dz^{(1)} \ldots dz^{(m)} \leq
\int_{\mathbb{T}^N}
\big\| f(mz)\big\|_X dz\,.
\]
\end{lemma}

\begin{proof} We fix some $m$, and  do induction with respect to $N$. For $N=1$ we obtain from Lemma \ref{lemma1} that
\begin{align*}
&
\underbrace{\int_{\mathbb{T}} \ldots  \int_{\mathbb{T}}}_{m-2}\int_{\mathbb{T}}\int_{\mathbb{T}}
\big\|  \underbrace{f
\big(
z^{(1)}+ \ldots + z^{(m-2)}+ z^{(m-1)}+ z^{(m)}
\big)}
_{=: g_{z^{(1)}, \ldots, z^{(m-2)}} (z^{(m-1)}+ z^{(m)})}
\big\|_X
dz^{(m-1)}dz^{(m)} dz^{(1)}\ldots dz^{(m-2)}
\\&
\leq
\underbrace{\int_{\mathbb{T}} \ldots  \int_{\mathbb{T}}}_{m-2}\int_{\mathbb{T}}
\big\|
g_{z^{(1)}, \ldots, z^{(m-2)}} (2w)
\big\|_X
dw \,\,dz^{(1)}\ldots dz^{(m-2)}
\\&
=
\underbrace{\int_{\mathbb{T}} \ldots  \int_{\mathbb{T}}}_{m-3} \int_{\mathbb{T}}\int_{\mathbb{T}}
\big\|
f\big(z^{(1)}+ \ldots + z^{(m-2)}+ 2w\big)
\big\|_X
dw  dz^{(m-2)} dz^{(1)}\ldots dz^{(m-3)}
\\&
\leq
\underbrace{\int_{\mathbb{T}} \ldots  \int_{\mathbb{T}}}_{m-3} \int_{\mathbb{T}}
\big\|
f\big(z^{(1)}+ \ldots + z^{(m-3)}+ 3w\big)
\big\|_X
dz^{(1)}\ldots dz^{(m-3)}\,\,dw
\\&
\leq \ldots \leq \int_{\mathbb{T}}
\big\| f(mz)\big\|_X dz\,.
\end{align*}
We now assume that the conclusion holds for $N-1$ and write each $z\in \mathbb{T}^N$ as $z=(u,w)$, with $u\in \mathbb{T}^{N-1}$ and $w\in \mathbb{T}$. Then, using the case $N=1$ in the first inequality and the inductive hypothesis in the second, we have
\begin{align*}
&
\int_{\mathbb{T}^N} \ldots  \int_{\mathbb{T}^{N}}
\big\| f\big(z^{(1)}+ \ldots + z^{(m)} \big)   \big\|_X dz^{(1)} \ldots dz^{(m)}
\\&
=\int_{\mathbb{T}^{N-1}} \ldots  \int_{\mathbb{T}^{N-1}}
\Big(
\int_{\mathbb{T}} \ldots  \int_{\mathbb{T}}
\big\| f\big((u^{(1)},w_1)+ \ldots + (u^{(m)},w_m) \big)   \big\|_X
dw_1 \ldots dw_N
\Big)\, du^{(1)} \ldots du^{(m)}
\\&
\le
\int_{\mathbb{T}^{N-1}} \ldots  \int_{\mathbb{T}^{N-1}}
\Big(
\int_{\mathbb{T}}
\big\| f\big((u^{(1)},mw)+ \ldots + (u^{(m)},mw) \big)   \big\|_X
dw
\Big)\, du^{(1)} \ldots du^{(m)}
\\&
=
\int_{\mathbb{T}}
\Big(
\int_{\mathbb{T}^{N-1}} \ldots  \int_{\mathbb{T}^{N-1}}
\big\| f\big((u^{(1)},mw)+ \ldots + (u^{(m)},mw) \big)   \big\|_X
 du^{(1)} \ldots du^{(m)}
\Big)\,dw
\\&
\leq
\int_{\mathbb{T}}
\Big(
\int_{\mathbb{T}^{N-1}}
\big\| f\big((mu,mw)+ \ldots + (mu,mw) \big)   \big\|_X
 du
\Big)\,dw
\\&
=
\int_{\mathbb{T}^N}
\big\| f(mz)\big\|_X dz\,,
\end{align*}
as desired.
\end{proof}

\vspace{3mm}

\noindent We are now ready to give the {\it proof of the inequality from Proposition \ref{maintool}}. By the polarization formula  we know that for every choice of $z_1^{(1)}, \ldots, z_m^{(m)} \in \mathbb{T}^N$
we have
\[
T\big(z^{(1)}, \ldots,z^{(m)}\big)
=
\frac{1}{2^mm!}\sum_{\varepsilon_i = \pm 1} \varepsilon_i \ldots \varepsilon_m
P\Big( \sum_{i=1}^N \varepsilon_i z^{(i)} \Big)
\]
(see e.g \cite{Di99} or
\cite{Fl99}). Hence we deduce from Lemma \ref{lemma2}
\begin{align*}
\int_{\mathbb{T}^N} \ldots \int_{\mathbb{T}^N} \big\| & T\big(z^{(1)}, \ldots,z^{(m)}\big)   \big\|_X
dz^{(1)} \ldots dz^{(m)}
\\&
\leq
\frac{1}{2^mm!}\sum_{\varepsilon_i = \pm 1}
\int_{\mathbb{T}^N} \ldots \int_{\mathbb{T}^N}
\Big\| P\Big( \sum_{i=1}^N \varepsilon_i z^{(i)} \Big)  \Big\|_X
dz^{(1)} \ldots dz^{(m)}
\\&
=
\frac{1}{2^mm!}\sum_{\varepsilon_i = \pm 1}
\int_{\mathbb{T}^N} \ldots \int_{\mathbb{T}^N}
\Big\| P\Big( \sum_{i=1}^N z^{(i)} \Big)  \Big\|_X
dz^{(1)} \ldots dz^{(m)}
\\&
=
\frac{1}{m!}
\int_{\mathbb{T}^N} \ldots \int_{\mathbb{T}^N}
\Big\| P\Big( \sum_{i=1}^N z^{(i)} \Big)  \Big\|_X
dz^{(1)} \ldots dz^{(m)}
\\&
\leq
\frac{1}{m!}
\int_{\mathbb{T}^N}
\big\| P\big( mz\big)  \big\|_X
dz
=
\frac{m^m}{m!}
\int_{\mathbb{T}^N}
\big\| P\big( z\big)  \big\|_X
dz\,.
\end{align*}
Then by Lemma \ref{lemma0} we obtain
\begin{align*}
\Big(\sum_{i_1, \ldots,i_m}^N \big\| a_{i_1, \ldots,i_m} \big\|_X^q\Big)^{1/q}
&
\leq \big(C_q(X)K\big)^m
 \int_{\mathbb{T}^\infty}  \ldots \int_{\mathbb{T}^\infty}
\big\|   T(z^{(1)}, \ldots,z^{(m)})     \big\|_X dz^{(1)} \ldots dz^{(m)}
\\&
=
\big(C_q(X)K\big)^m
\frac{m^m}{m!}
\int_{\mathbb{T}^N}
\big\| P\big( z\big)  \big\|_X
dz\,,
\end{align*}
which completes the proof of Proposition \ref{maintool}. $\Box$

\vspace{3mm}

A second proposition is needed which allows to reduce  the proof of our main result \ref{MAIN} to the
homogeneous case. It is
a vector-valued version of a result of \cite[Theorem 9.2]{CoGa86} with a similar proof (here only given for the sake of completeness).

\begin{proposition} \label{lemma3}
There is a contractive projection
\[
\Phi_m: H_p(\mathbb{T}^N, X) \rightarrow H_p(\mathbb{T}^N, X)\,,\,\,\, f \mapsto f_m,
\]
such for all $f \in H_p(\mathbb{T}^N, X)$
\begin{equation} \label{alpha}
\hat{f}(\alpha) = \hat{f}_m(\alpha) \,\, \text{ for all  } \,\, \alpha \in \mathbb{N}_0^N
\,\, \text{ with  } \,\, |\alpha|=m\,.
\end{equation}
\end{proposition}

\begin{proof}
Let $\mathcal{P}(\mathbb{C}^N,X) \subset H_p(\mathbb{T}^N, X)$ be the  subspace all finite polynomials
$f= \sum_{\alpha \in \Lambda} c_\alpha z^\alpha$; here $\Lambda$ is a finite set of multi indices in $\mathbb{N}_0^N$ and the coefficients $c_\alpha \in X$. Define the linear projection $\Phi^0_m$ on
$\mathcal{P}(\mathbb{C}^N,X)$ by
\begin{equation*}
\Phi^0_m(f)(z)=f_m(z)= \sum_{\alpha \in \Lambda, |\alpha|=m} \hat{f}(\alpha)z^\alpha\,;
\end{equation*}
clearly, we have \eqref{alpha}.
In order to show that $\Phi^0_m$ is a contraction on $\big(\mathcal{P}(\mathbb{C}^N,X), \|\cdot\|_p\big)$
fix some function $f \in \mathcal{P}(\mathbb{C}^N,X)$ and $z \in \mathbb{T}^N$, and define
\[
f(z \cdot): \mathbb{T} \rightarrow X\,, \,\,\, w\mapsto f(z w)\,.
\]
Clearly, we have
\[
f(z w) = \sum_{k} f_k(z)w^k\,,
\]
and hence
\[
f_m(z) = \int_{\mathbb{T}} f(zw) w^{-m} dw\,.
\]
Integration,  the continuous Minkowski inequality and the rotation invariance of the normalized Lebesgue measure on $\mathbb{T}^N$
give
\begin{multline*}
\int_{\mathbb{T}^N} \big\| f_m(z) \big\|_X^p dz
=\int_{\mathbb{T}^N} \big\|\int_{\mathbb{T}} f(zw) w^{-m} dw  \big\|_X^p dz\\
\le\int_{\mathbb{T}^N} \Big( \int_{\mathbb{T}} \big\| f(zw) \big\|_X dw\Big)^p dz
\le\int_{\mathbb{T}} \int_{\mathbb{T}^N} \big\| f(zw) \big\|_X^p dz dw
=\int_{\mathbb{T}^N} \big\| f(z) \big\|_X^p dz\,,
\end{multline*}
which proves that   $\Phi^0_m$ is a contraction on $(\mathcal{P}(\mathbb{C}^N,X), \|\cdot\|_p)$. By Fejer's theorem (vector-valued) we know that $\mathcal{P}(\mathbb{C}^N,X)$ is a dense subspace of
$ H_p(\mathbb{T}^N, X)$. Hence $\Phi^0_m$ extends to a contractive projection $\Phi_m$ on $ H_p(\mathbb{T}^N, X)$.
This extension $\Phi_m$ still satisfies \eqref{alpha} since for each multi index $\alpha $ the mapping
$ H_p(\mathbb{T}^N, X) \rightarrow X, \,\,\, f \mapsto \hat{f}(\alpha)$ is continuous.
\end{proof}

\section{Proof of the main result }
\noindent We are now ready to prove Theorem \ref{MAIN}.
Let $1 \leq p < \infty$, and recall from \eqref{MathAnn}  that
\[
 1- \dfrac{1}{\ct(X)}=  S_\infty(X) \leq S_p(X) \,;
\]
see Remark \ref{dani} for a direct argument. Hence it suffices to concentrate on the upper estimate in Theorem \ref{MAIN}: Since we obviously have $  S_p(X) \leq  S_1(X)$, we are going to prove that
\begin{equation}\label{eqS1}
S_1(X) \leq  1- \dfrac{1}{\ct(X)}\,.
\end{equation}

Suppose first that $X$ has no finite cotype.
For $D=\sum_n a_n n^{-s}  \in \mathcal{H}_1(X)$ we take
 $f \in H_1(\mathbb{T}^\infty, X)$ with $D=  \mathfrak{B}_Xf$. Note that
\[ | \hat{f}(\alpha) |\le \int_{\mathbb{T}^{\infty}} |f(w) w^{- \alpha}| dw=
\|f\|_{L_{1} (\mathbb{T}^{\infty}, X)}<\infty \,
 \]
and, by the definition of $\mathfrak{B}_X$, the coefficients of $D$
are also bounded by $\|f\|_{L_{1} (\mathbb{T}^{\infty}, X)}$. As a consequence,
\[
\sum_{n=1}^\infty \|a_n\|_X \frac{1}{n^s} \le \sum_{n=1}^\infty
\|f\|_{L_{1} (\mathbb{T}^{\infty}, X)} \frac{1}{n^s}<\infty
\]whenever $\re s > 1$. This means that $S_1(X)\le 1$ and gives \eqref{eqS1} for $\ct(X)=\infty$.

Now if  $X$ has finite cotype, take $q > \ct(X)$ and $\varepsilon > 0$, and put
$s= \big(1-\frac{1}{q} \big) \big(1+ 2\varepsilon \big)$. Choose an integer  $k_0$ such $p_{k_0}^{\varepsilon / q'} > eC_q(X)  K \sum_{j=1}^\infty \frac{1}{p_{j}^{1+ \varepsilon}}$, and define
\[
\tilde{p} = (\underbrace{p_{k_0}, \ldots,p_{k_0}}_{k_0\,\,\, \text{times}}, p_{k_0+1}, p_{k_0+2}, \ldots ).
\]
We are going to show that there is a constant $C(q,X,\varepsilon)>0$ such that for every $f \in H_1(\mathbb{T}^\infty,X)$ we have
\begin{align} \label{infy}
\sum_{\alpha \in \mathbb{N}_0^{(\mathbb{N})}} \|\hat{f}(\alpha) \|_X \frac{1}{\tilde{p}^{s \alpha}} \leq C(q,X,\varepsilon)\|f\|_{H_1(\mathbb{T}^\infty,X)}.
\end{align}
This finishes the argument: By Remark \ref{remark} the sequence $1/p^s \in \mon H_1(\mathbb{T}^\infty, X)$. But in view of Bohr's transform from \eqref{vision},
this means that for every Dirichlet series $D=\sum_n a_n n^{-s} =  \mathfrak{B}_Xf \in \mathcal{H}_1(X)$ with $f \in H_1(\mathbb{T}^\infty, X)$ we have
\[
\sum_{n=1}^\infty \|a_n\|_X \frac{1}{n^s} = \sum_{\alpha \in \mathbb{N}_0^{(\mathbb{N})}} \|\hat{f}(\alpha) \|_X \frac{1}{p^{s \alpha}} < \infty\,.
\]
Therefore $\sigma_a(D) \le \big(1-\frac{1}{q} \big) \big(1+ 2\varepsilon \big)$  for each such $D$ which, since $\varepsilon > 0$
was arbitrary, is what we wanted to prove.

It remains to check  \eqref{infy}; the idea is to show first that \eqref{infy} holds for all
$X$-valued $H_1$-functions which only depend on $N$ variables: There is a constant $C(q,X,\varepsilon)>0$ such that for
 all $N$ and every $f \in H_1(\mathbb{T}^N,X)$ we have
\begin{align} \label{finy}
\sum_{\alpha \in \mathbb{N}_0^N}\|\hat{f}(\alpha) \|_X \frac{1}{\tilde{p}^{s \alpha}} \leq C(q,X,\varepsilon)\|f\|_{H_1(\mathbb{T}^N,X)}.
\end{align}
In order to understand that \eqref{finy} implies \eqref{infy} (and hence the conclusion), assume that \eqref{finy} holds and take some
$f \in H_1(\mathbb{T}^\infty,X)$. Given an arbitrary $N$, define
\begin{align*}
f_N: \mathbb{T}^N \rightarrow X,\,\, f_N(w) = \int_{\mathbb{T}^\infty} f(w, \tilde{w}) d\tilde{w}.
\end{align*}
Then it can be easily shown that $f_N \in L_1(\mathbb{T}^N,X)$, $\|f_N\|_1 \le \|f\|_1$, and
$\hat{f_N}(\alpha)=\hat{f}(\alpha)$ for all $\alpha \in \mathbb{Z}^N$. If we now apply \eqref{finy} to this $f_N$,
we get \begin{align*}
\sum_{\alpha \in \mathbb{N}_0^N}\|\hat{f}(\alpha) \|_X \frac{1}{\tilde{p}^{s \alpha}} \leq C(q,X,\varepsilon)\|f\|_{H_1(\mathbb{T}^\infty,X)}\,,
\end{align*}
which, after taking the supremum over all possible $N$ on the left side, leads to \eqref{infy}.

We turn to the proof of  \eqref{finy}, and  here in a first step will show the following: For
every $N$, every $m$-homogeneous polynomial $P: \mathbb{C}^N \rightarrow X$ and every $u \in \ell_{q'}$
we have
\begin{align} \label{poly}
\sum_{\substack{ \alpha \in  \mathbb{N}_0^N \\ |\alpha|=m}} \|\hat{P}(\alpha) u^\alpha\|_X
\leq \big(eC_q(X)K \big)^m  \int_{\mathbb{T}^N} \big\|P(z)  \big\|_X dz \,\, \Big( \sum_{j=1}^\infty |u_j|^{q'} \Big)^{m/q'}\,.
\end{align}
Indeed, take such a polynomial $P(z)=\sum_{\alpha \in  \mathbb{N}_0^N,|\alpha|=m} \hat{P}(\alpha) z^\alpha, \, z \in \mathbb{T}^N$, and look at its unique $m$-linear symmetrization
\[
T: \mathbb{C}^N \times \ldots \times  \mathbb{C}^N \rightarrow X,\,\,\,
 T(z^{(1)}, \ldots,z^{(m)}) = \sum_{i_1, \ldots,i_m=1}^N a_{i_1, \ldots,i_m} z^{(1)}_{i_1}, \ldots,z^{(m)}_{i_m}\,.
\]
Then we know from Proposition  \ref{maintool} that
\begin{align*}
 \Big(\sum_{i_1, \ldots,i_m=1}^N \big\|a_{i_1, \ldots,i_m}\big\|_X^q \Big)^{1/q}
\leq
\big(eC_q(X) K\big)^{m}   \int_{\mathbb{T}^N} \big\| P(z) \big\|_X dz\,.
\end{align*}
Hence  \eqref{poly} follows by H\"older's inequality:
\begin{multline*}
\sum_{\alpha \in  \mathbb{N}_0^N,|\alpha|=m} \big\| \hat{P}(\alpha) u^\alpha \big\|_X
=\sum_{i_1, \ldots,i_m=1}^N \big\|a_{i_1, \ldots,i_m}\big\|_X
|u_{i_1} \ldots u_{i_N}|
\\
\leq
\big(eC_q(X) K\big)^{m}   \int_{\mathbb{T}^N} \big\| P(z) \big\|_X dz
\,\, \Big( \sum_{j=1}^\infty |u_j|^{q'} \Big)^{m/q'}\,.
\end{multline*}
We finally give the proof of \eqref{finy}: Take $f \in H_1(\mathbb{T}^N, X)$, and recall from Proposition \ref{lemma3}
that for each integer $m$ there is an $m$-homogeneous polynomial $P_m: \mathbb{C}^N \rightarrow X$ such that
$\|P_m\|_{H_1(\mathbb{T}^N, X)} \leq  \|f\|_{H_1(\mathbb{T}^N, X)}$ and $\hat{P}_m(\alpha)=\hat{f}(\alpha)$
for all $\alpha \in  \mathbb{N}_0^N$ with $|\alpha|=m$. Finally, from \eqref{poly}, the definition of $s$, and the fact that  $\max\{p_{k_0},p_j\} \leq \tilde{p}_j$ for all $j$ we conclude that
\begin{align*} \label{finally}
\sum_{\alpha \in \mathbb{N}_0^N}\|\hat{f}(\alpha) \|_X \frac{1}{\tilde{p}^{s \alpha}}
&
=\sum_{m=1}^\infty\sum_{\alpha \in \mathbb{N}_0^N, |\alpha|=m}\|\hat{P}_m(\alpha) \|_X \frac{1}{\tilde{p}^{s \alpha}}
 \\[1ex]&
\leq \sum_{m=1}^\infty \big(eC_q(X) K\big)^{m}\big\|P_m\big\|_{H_1(\mathbb{T}^N,X)}\,\, \Big( \sum_{j=1}^\infty \frac{1}{\tilde{p}_j^{sq'}} \Big)^{m/q'}
 \\[1ex]&
= \sum_{m=1}^\infty \big(eC_q(X) K\big)^{m}\big\|f\big\|_{H_1(\mathbb{T}^N,X)}\,\, \Big( \sum_{j=1}^\infty \frac{1}{\tilde{p}_j^{1+2\varepsilon}} \Big)^{m/q'}
\\[1ex]&
= \sum_{m=1}^\infty  \big(eC_q(X) K\big)^{m}\big\|f\big\|_{H_1(\mathbb{T}^N,X)}\,\,
\Big( \sum_{j=1}^\infty \frac{1}{\tilde{p}_j^{1+\varepsilon}}
\frac{1}{\tilde{p}_j^{\varepsilon}}
\Big)^{m/q'}
\\[1ex]&
\le
\big\|f\big\|_{H_1(\mathbb{T}^N,X)}
\sum_{m=1}^\infty
\Bigg(
\underbrace{\frac{eC_q(X) K   \Big( \sum_{j=1}^\infty \frac{1}{p_j^{1+\varepsilon}} \Big)^{1/1+\varepsilon}       }{p_{k_0}^{\varepsilon/q'}}}_{\le 1}
\Bigg)^{m}\,.
\end{align*}
 This completes the proof of Theorem \ref{MAIN}.\,\,\,\,$\Box$

\begin{remark} \label{dani}
We end this note with a direct proof of the fact \begin{equation}
\label{ineq} 1- \dfrac{1}{\ct(X)}\leq S_p(X) \,,\,\ 1 \leq p < \infty
\end{equation}
in which we do not use the inequality
 \begin{equation}
\label{tt} 1- \dfrac{1}{\ct(X)}\leq S_\infty(X)
\end{equation}
from \cite{DeGaMaPG08}
(here repeated in \eqref{MathAnn}). The proof of \eqref{tt} given in \cite{DeGaMaPG08} in a first step shows that $1- 1/\Pi(X) \le S_\infty(X)$ where $$\Pi(X) = \inf \big\{   r \ge 2 \,|\, \mathrm{id}_X
\, \text{ is  }\,  (r,1)-\text{summing} \big\},$$ and then, in a second step,  applies a fundamental theorem of
Maurey and Pisier stating that $\Pi(X)= \ct(X)$.

The following argument for \eqref{ineq} is very similar to the orginal one from  \cite{DeGaMaPG08} but does not use  the Maurey-Pisier theorem (since we here  consider $\mathcal{H}_p(X), 1 \le p < \infty$ instead of $\mathcal{H}_\infty(X)$):
  By the proof of Corollary~\ref{Bohri}, inequality \eqref{ineq} is equivalent to \begin{equation*}\label{ineq2}M_p(X)\le \frac{\ct(X)}{\ct(X)-1}\,.\end{equation*}
Take $r< M_p(X)$, so that $\ell_r \cap B_{c_0} \subset  \mon H_{p}(\mathbb{T}^{\infty},X)$.
Let $H_p^1(\mathbb{T}^\infty, X)$ be the subspace of $H_{p}(\mathbb{T}^{\infty},X)$ formed by all 1-homogeneous polynomials (i.e., linear operators). We can define a bilinear operator $\ell_r\times H_p^1(\mathbb{T}^\infty, X) \to \ell_1(X)$ by $(z,f)\mapsto (z_jf(e_j))_j$ which, by a closed graph argument, is continuous. Therefore, there is a constant $M$ such that
for all $z \in \ell_r$ and all $f \in H_p^1(\mathbb{T}^\infty, X)$ we have
$$\sum_j|z_j|\|f(e_j)\|_X\le M \|z\|_{\ell_r} \|f\|_{H_{p}(\mathbb{T}^{\infty},X)}.$$ Taking the supremum over all $z\in B_{\ell_r}$ we obtain
for all $f \in H_p^1(\mathbb{T}^\infty, X)$
$$\big(\sum_j\|f(e_j)\|_X^{r'}\big)^{1/r'}\le M  \|f\|_{H_{p}(\mathbb{T}^{\infty},X)}.$$
Now, take $x_1,\dots,x_N\in X$ and define $f\in H_p^1(\mathbb{T}^\infty, X)$  by $f(e_j)=x_j$ if $1\le j\le N$, $f(e_j)=0$ if $j>N$ and extend it by linearity. By the previous inequality and Lemma~\ref{lemma3} we have
$$\big(\sum_{j=1}^N\|x_j\|_X^{r'}\big)^{1/r'}\le M  \Big( \int_{\mathbb{T}^N} \Big\| \sum_{j=1}^N  x_j z_j  \Big\|_X^{r'} dz \Big)^{1/r'}\,.$$By Kahane's inequality, $X$ has cotype $r'$, which means that $r'>\ct(X)$ or, equivalently, $r<\frac{\ct(X)}{\ct(X)-1}$. Since $r< M_p(X)$ was arbitrary, we obtain \eqref{ineq}.

\end{remark}

\bibliographystyle{abbrv}
%

\noindent dcarando@dm.uba.ar \\  defant@mathematik.uni-oldenburg.de \\  psevilla@mat.upv.es

\end{document}